\documentclass[a4paper, oneside, 11pt]{article}
\usepackage[english]{babel}
\usepackage{latexsym,amsfonts,amsmath,amssymb,enumerate,graphics,enumerate,amsthm}
\usepackage{tikz,wrapfig}
\usepackage{fontawesome}
\title{\textbf{Mean equicontinuous factor maps} }
\author{Till Hauser\thanks{
{\small\faEnvelopeO}~hauser.math@mail.de\\
Funded by the Deutsche Forschungsgemeinschaft (DFG, German Research Foundation) – 530703788.}}

\usepackage{float}

\let\epsilon=\varepsilon

\theoremstyle{definition}
\newtheorem{definition}{Definition}[section]
\newtheorem{theorem}[definition]{Theorem}
\newtheorem*{theorem*}{Theorem}
\newtheorem{innercustomthm}{Theorem}
\newenvironment{theorem**}[1]
  {\renewcommand\theinnercustomthm{#1}\innercustomthm}
  {\endinnercustomthm}

\newtheorem{proposition}[definition]{Proposition}
\newtheorem{lemma}[definition]{Lemma}
\newtheorem{corollary}[definition]{Corollary}

\theoremstyle{remark}
\newtheorem{remark}[definition]{Remark}
\newtheorem{example}[definition]{Example}
\newtheorem*{acknowledgement}{Acknowledgement}

\begin{document}
\maketitle

\begin{abstract}
Mean equicontinity is a well studied notion for actions. We propose a definition of mean equicontinuous factor maps that generalizes mean equicontinuity to the relative context. For this we work in the context of countable amenable groups. We show that a factor map is equicontinuous, if and only if it is mean equicontinuous and distal. Furthermore, we show that a factor map is topo-isomorphic, if and only if it is mean equicontinuous and proximal. We present that the notions of topo-isomorphy and Banach proximality coincide for all factor maps. In the second part of the paper we turn our attention to decomposition and composition properties. It is well known that a mean equicontinuous action is a topo-isomorphic extension of an equicontinuous action. In the context of minimal and the context of weakly mean equicontinuous actions, respectively, we show that any mean equicontinuous factor map can be decomposed into an equicontinuous factor map after a topo-isomorphic factor map. Furthermore, for factor maps between weakly mean equicontinuous actions we show that a factor map is mean equicontinuous, if and only if it is the composition of an equicontinuous factor map after a topo-isomorphic factor map.  We will see that this decomposition is always unique up to conjugacy. 
\end{abstract}
\maketitle

\section{Introduction}\label{sec:introduction}

	Fundamental concepts in the theory of topological dynamical systems are given by the notions of equicontinuity, distality and proximality. Since any action can be interpreted as a factor map onto the trivial action on one point action, it is natural to define these notions not only for actions, but also for factor maps. This type of generalization allows for strong structural results, like for example the celebrated Furstenberg structure theorem \cite{auslander1988minimal, ellis1978furstenberg, furstenberg1963structure}.

	A extensively studied notion in the context of topological dynamics is the notion of a \emph{mean equicontinuous} action. It was introduced under the name of \emph{mean-L-stability} by Fomin \cite{fomin1951dynamical} and is also called \emph{equicontinuity in the mean}. See \cite{auslander1959mean, auslander1980interval, cai2022properties, downarowicz2016isomorphic, downarowicz2021all, fomin1951dynamical, fuhrmann2023continuity, fuhrmann2022structure, garcia2017weak, garcia2021mean, garcia2019mean, huang2021bounded, li2015mean, li2018chaotic, qiu2020note, scarpellini1982stability, xu2022weak, yang2024anote} for some of the relevant references.  
	We recommend \cite{li2021mean} for an overview over the recent developments of the theory of mean equicontinuous actions. 
	
	It is natural to ask, whether this notion can be generalized to a notion for factor maps. We will present below how this can be achieved. 
For this we will work in the context of countable (discrete) amenable groups following
	\cite{fuhrmann2023continuity, fuhrmann2022structure, garcia2017weak, lkacka2018quasi, xu2022weak}. 
	Let us recall the definition of mean equicontinuity for actions. We denote $\Delta$ for the symmetric difference of sets and $|\cdot|$ for the cardinality. 
The \emph{amenability} of a countable group is characterized by the existence of a F\o lner sequence $(F_n)_{n\in \mathbb{N}}$, i.e.\ a sequence of finite sets $F_n$ for which 
	$|F_n\Delta gF_n|/|F_n|\to 0$ for all $g\in G$. 	
	For an action of a countable amenable group $G$ on a compact metric space $(X,d)$, we introduce the \emph{Weyl pseudometric} as 
	\[D(x,x'):=\sup_{\mathcal{F}} \limsup_{n\to \infty}\frac{1}{|F_n|}\sum_{g\in F_n}d(g.x,g.x'),\] where the supremum is taken over all F\o lner sequences $\mathcal{F}=(F_n)_{n\in \mathbb{N}}$. As shown in \cite[Lemma 7]{lkacka2018quasi}, we have 
	\[D(x,x')=\limsup_{n\to \infty}\sup_{g'\in G}\frac{1}{|F_n|}\sum_{g\in F_n}d(gg'.x,gg'.x')\] for any F\o lner sequence $(F_n)$. See \cite{lkacka2018quasi} for further definitions of equivalent pseudo metrics. 	
	It is straight forward to observe that for equivalent metrics on $X$, we obtain equivalent Weyl pseudometrics.
	An action is called \emph{(Weyl) mean equicontinuous}, if and only if for any $\epsilon>0$ there exists $\delta>0$ such that for all $x,x'\in X$ we have that $d(x,x')<\delta$ implies $D(x,x')<\epsilon$. Since all continuous metrics on $X$ are equivalent this notion does not depend on the choice of a continuous metric on $X$. 
	
	Now consider a factor map $\pi\colon X\to Y$ and denote $R(\pi):=\{(x,x')\in X^2;\, \pi(x)=\pi(x')\}$. In order to define mean equicontinuity for $\pi$ the following property appears to be a natural candidate.
	
\begin{itemize}
	\item[(M)] For $\epsilon>0$ there exists $\delta>0$ such that all $(x,x')\in R(\pi)$ with $d(x,x')<\delta$ satisfy $D(x,x')<\epsilon$.
\end{itemize}
	Clearly also (M) is independent from the choice of a continuous metric. 
	Note that an action is mean equicontinuous, if and only if the factor map onto one point satisfies (M). Furthermore, it is straight forward to observe that any equicontinuous factor map satisfies (M). 
	However, it turns out that (M) is to weak for a nice structure theory. To shed some light on this we need the notion of a Banach proximal pair. 	 	
	
	We call a pair $(x,x')\in X^2$ \emph{Banach proximal}, whenever $D(x,x')=0$. Clearly also this notion does not depend on the choice of a continuous metric on $X$. A pair $(x,x')$ is Banach proximal, if and only if the set $\{g\in G;\, d(g.x,g.x')<\epsilon\}$ has (lower) Banach density $1$ for any $\epsilon>0$. See \cite[Section 3]{lkacka2018quasi} for details on Banach densities and \cite[Corollary 8]{lkacka2018quasi} for reference. It follows that this definition is equivalent to the one considered in \cite{li2011chaos, li2014proximality, li2015mean}. We denote $BP$ for the \emph{Banach proximal} relation, i.e.\ the set of all Banach proximal pairs. 
	Recall from \cite[Theorem 3.5]{li2015mean} that for any mean equicontinuous action $X$ the Banach proximal relation $BP$ and the regionally proximal relation $RP$ coincide. 
	Let us denote $BP(\pi):=R(\pi)\cap BP$ for the \emph{$\pi$-Banach proximal relation}. 
	We will see in Example \ref{exa:M2doesNotImplyM1} below, that (M) is not sufficient to ensure that $BP(\pi)$ and the $\pi$-regionally proximal relation $RP(\pi)$ coincide. 
	We thus aim for a property stronger than (M) to define mean equicontinuous factor maps. 
	
	Exploring the pseudometric properties of $D$ it is straight forward to observe that $X$ is mean equicontinuous, if and only if $D\in C(X^2)$. Another natural candidate for a generalization thus would be to ask for $D\in C(R(\pi))$. However, in Example \ref{exa:Dnotcontinuous} below, we will see that even for equicontinuous extensions, we do not necessarily have $D\in C(R(\pi))$. Thus this condition is to strong. 
	
	The key to find the right property is to reformulate (M). 
	For this we call a sequence $(x_n,x_n')$ in $X^2$ \emph{asymptotically Banach proximal}, whenever $D(x_n,x_n')\to 0$. Clearly also this notion is independent from the choice of a continuous metric in $X$. 
	We will see in Proposition \ref{pro:M1characterization} below, that (M) is satisfied, if and only if for convergent sequences in $R(\pi)$ the Banach proximality of the limit implies the asymptotical Banach proximality of the sequence. We propose the following definition of a mean equicontinuous factor map. 
	
\begin{definition}
	A factor map $\pi\colon X\to Y$ between actions of a countable amenable group $G$ is called \emph{mean equicontinuous}, whenever for convergent sequences in $R(\pi)$ the asymptotical Banach proximality of the sequence is equivalent to the Banach proximality of the limit. 
\end{definition}

	Since for actions $X$ mean equicontinuity is equivalent to $D\in C(X^2)$ it is straight forward to observe that an action $X$ is mean equicontinuous, if and only if the factor map onto a point is mean equicontinuous. Furthermore, we will see in Theorem \ref{the:char_equicontinuity} that any equicontinuous factor map is mean equicontinuous. This justifies the choice of such a notion. As already mentioned, we will show below (in  Theorem \ref{the:MEimpliesRPisRBPisPisBPisICER}) that this definition allows us to show the following generalization of \cite[Theorem 3.5]{li2015mean}. 
	 
\begin{theorem*}
	For any mean equicontinuous factor map $\pi$ we have that $RP(\pi)=BP(\pi)$ is a closed invariant equivalence relation. 
\end{theorem*}	

	Any equicontinuous factor map is mean equicontinuous. Another important class of mean equicontinuous factor maps is given by Banach proximal factor maps, which we introduce next. A factor map $\pi\colon X\to Y$ is called \emph{Banach proximal}, whenever $BP(\pi)=R(\pi)$, i.e.\ whenever all pairs $(x,x')\in X^2$ with $\pi(x)=\pi(x')$ satisfy $D(x,x')=0$ \cite{li2015mean}. 
	Note that a factor map $\pi\colon X\to Y$ is called \emph{topo-isomorphic}, whenever for each invariant Borel probability measure $\mu$ on $X$ the map $\pi$ induces a measure theoretic isomorphism between $X$ equipped with $\mu$ and $Y$ equipped with the push forward of $\mu$ under $\pi$.
	In \cite[Theorem 4.3]{qiu2020note} and in \cite[Chapter 3]{fuhrmann2022structure} it was shown that whenever $\pi\colon X\to Y$ is a factor map onto an equicontinuous action, then $\pi$ is Banach proximal, if and only if it is topo-isomorphic. 
	We will see in Theorem \ref{the:topo_iso_and_Banach_prox} that Banach proximality and topo-isomorphy are equivalent concepts for all factor maps and hence that any topo-isomorphic factor map is mean equicontinuous. Since any Banach proximal pair is proximal, we also observe that topo-isomorphic factor maps are proximal. In fact this observation can be used to characterize topo-isomorphic factor maps. 
	
\begin{theorem**}{\ref{the:characterization_Banach_proximal_factor_maps}}
	For a factor map $\pi$ the following statements are equivalent.
	\begin{itemize}
	\item[(i)] $\pi$ is topo-isomorphic.
	\item[(ii)] $\pi$ is mean equicontinuous and proximal. 
	\item[(iii)] $\pi$ is mean equicontinuous and $R(\pi)=RP(\pi)$. 
	\end{itemize}
\end{theorem**}

	It was shown in \cite[Corollary 3.6]{li2015mean} that an action (of $\mathbb{Z}$) is equicontinuous, if and only if it is mean equicontinuous and distal. We will see in Theorem \ref{the:char_equicontinuity} the following interesting parallel between equicontinuous and topo-isomorphic factor maps. 

\begin{theorem*}
A factor map is equicontinuous, if and only if it is mean equicontinuous and distal. 
\end{theorem*}

	In the study of topological dynamics it is natural to relate actions to equicontinuous ones. For an action one can consider all equicontinuous factors and it is well known that among them there exists a maximal equicontinuous factor \cite{auslander1988minimal}. It was shown in \cite{downarowicz2016isomorphic, fuhrmann2022structure, li2015mean}
	that an action $X$ is mean equicontinuous, if and only if the factor map $\phi\colon X\to \mathbb{T}$ onto it's maximal equicontinuous factor $\mathbb{T}$ is a topo-isomorphic factor map.
	
	 Reformulating this result, we observe that any mean equicontinuous factor map $\pi\colon X\to \{0\}$ onto one point can be decomposed into $\pi=\psi\circ \phi$, with $\phi\colon X\to \mathbb{T}$ topo-isomorphic and $\psi\colon \mathbb{T}\to \{0\}$ equicontinuous. 
	It is natural to ask, whether such a decomposition is possible in general. 
	In Section \ref{sec:decomposition} we will see the following. 
	
\begin{theorem**}{\ref{the:decompositionMinimal}}
	Any mean equicontinuous factor map $\pi\colon X\to Y$ between minimal actions decomposes as $\pi=\psi\circ \phi$ into a topo-isomorphic factor map $\phi\colon X\to X\big/BP(\pi)$ and an equicontinuous factor map $\psi\colon X\big/BP(\pi)\to Y$.
\end{theorem**}

	In order to go beyond the minimal case we consider weakly mean equicontinuous actions \cite{fuhrmann2023continuity, garcia2021mean, lkacka2018quasi, oxtoby1952ergodic, xu2022weak, yang2024anote, zheng2020new}. 
	 An action is called \emph{weakly mean equicontinuous}, whenever each $x\in X$ allows for only one invariant Borel probability measure $\mu_x$ on its orbit closure and the mapping $x\mapsto\mu_x$ is continuous w.r.t.\ the weak*-topology.  
	  Note that such actions are also called \emph{continuously pointwise ergodic} \cite{downarowicz2021all}. 
	 Trivially any uniquely ergodic action is weakly mean equicontinuous. 	  
Furthermore, any mean equicontinuous action is weakly mean equicontinuous. This was shown by \cite{fomin1951dynamical} for transitive actions. The general statement can be found in \cite{fuhrmann2022structure, xu2022weak} and the references within. 
	  See \cite[Example 5.1]{downarowicz2021all} for an example of a weakly mean equicontinuous action that is not mean equicontinuous and allows for more than one invariant Borel probability measure. Note that any minimal action that allows for at least two ergodic invariant Borel probability measures is not weakly mean equicontinuous. Such actions exist as demonstrated for example in \cite{downarowicz1991choquet}. Thus the following theorem is neither a generalization, nor a special case of Theorem \ref{the:decompositionMinimal}.  
	
\begin{theorem**}{\ref{the:decompositionWME}}
	Any mean equicontinuous factor map $\pi\colon X\to Y$ between weakly mean equicontinuous actions decomposes as $\pi=\psi\circ \phi$ into a topo-isomorphic factor map $\phi\colon X\to X\big/BP(\pi)$ and an equicontinuous factor map $\psi\colon X\big/BP(\pi)\to Y$.
\end{theorem**}

	We will furthermore see in Theorem \ref{the:uniqueness_factorization} that whenever a mean equicontinuous map decomposes in this way the factor maps are uniquely determined up to conjugacy. 
Recall that equicontinuous factor maps preserve the topological entropy \cite{yan2015conditional}. Because of the variational principle also topo-isomorphic factor maps preserve the topological entropy. We thus have the following. 

\begin{corollary}\label{cor:entropy}
Assume that $X$ is minimal or weakly mean equicontinuous. 
	If $\pi\colon X\to Y$ is a mean equicontinuous factor map, then $X$ and $Y$ have the same topological entropy. 
\end{corollary}

Unfortunately, it remains open, whether the decomposition presented in the Theorems \ref{the:decompositionMinimal} and \ref{the:decompositionWME} works for general mean equicontinuous factor maps. 
	In Section \ref{sec:combiningFactorMaps} we turn to the opposite question. What happens, if one composes topo-isomorphic and equicontinuous factor maps? 
	
\begin{theorem**}{\ref{the:WMEcomposingtiandEq}}
Assume that $X$ is weakly mean equicontinuous.
	If $\phi\colon X\to Y$ is a topo-isomorphy and $\psi\colon Y\to Z$ is mean equicontinuous, then $\psi\circ \phi\colon X\to Z$ is mean equicontinuous. 
\end{theorem**}

We will see in Example \ref{exa:topo-isometricAfterEquicont} that the order in the composition is important. We will demonstrate that there exist factor maps $\phi\colon X\to Y$ and $\psi\colon Y\to Z$ between minimal and uniquely ergodic actions of $\mathbb{Z}$ such that $\phi$ is equicontinuous, $\psi$ is topo-isomorphic and $\psi\circ \phi$ is not mean equicontinous. Unfortunately, it remains open, whether a similar result to Theorem \ref{the:WMEcomposingtiandEq} holds  for factor maps, even in the context of minimal actions.

\section{Preliminaries}\label{sec:preliminaries}

A pseudometric $d$ on a compact topological space $X$ is called \emph{continuous}, whenever $d\in C(X^2)$. 
Note that $X$ is metrizable, if and only if there exists a continuous metric $d$ on $X$. In this case the metric $d$ yields the topology of $X$.
For a continuous surjection $\pi\colon X\to Y$ between compact metrizable spaces, we denote $\hat{\pi}$ for the map $X^2\to Y^2$ with $\hat{\pi}(x,x'):=(\pi(x),\pi(x'))$. We denote $\Delta_X:=\{(x,x);\, x\in X\}$ for the \emph{diagonal of $X$}.  We denote $\overline{A}$ for the closure of a set $A\subseteq X$. 

For a compact metric space $(X,d)$ we equip $X$ with the Borel $\sigma$-algebra and write $\mathcal{M}(X)$ for the set of all Borel probability measures equipped with the weak*-topology. By the Banach–Alaoglu theorem $\mathcal{M}(X)$ is compact. 	

\subsection{Actions}

	Consider a countable group $G$ and a compact metric space $(X,d)$.  
	An \emph{action} of $G$ on $X$ is a group-homomorphism $\alpha$ from $G$ into the set of homeomorphisms $X\to X$. 
	If there is no need of a specified symbol $\alpha$ we write $g.x:=\alpha(g)(x)$ and say that $G$ \emph{acts} on $X$.   For $A\subseteq X$ we denote $A.x:=\{g.x;\, g\in A\}$. The \emph{orbit of $x$} is given by 
	$G.x$ and its \emph{orbit closure} is $\overline{G.x}$. 
	
	For $g\in G$ and $\mu\in \mathcal{M}(X)$ we denote 
	$g.\mu(A):=\mu(g^{-1}.A)$ for any Borel measurable $A\subseteq X$. 
	A measure $\mu\in \mathcal{M}(X)$ is called \emph{invariant}, if $g.\mu=\mu$ for all $g\in G$.
	An invariant $\mu\in \mathcal{M}(X)$ is said to be ergodic, if $\mu(A)\in \{0,1\}$ for all Borel subsets $A\subseteq G$ with $g.A=A$. 
	We denote $\mathcal{M}_G(X)$ and $\mathcal{M}_G^e(X)$ 
	for the set of invariant and the set of ergodic 
	Borel probability measures respectively.  
	Consider a F\o lner sequence $\mathcal{F}=(F_n)$. A point $x\in X$ is called \emph{$\mathcal{F}$-generic} for some $\mu\in \mathcal{M}_G(X)$, whenever $1/|F_n|\sum_{g\in F_n} g.\delta_x\to\mu$.	
\begin{remark}\label{rem:genericityCreation}
	For any action $X$ and any $\mu\in \mathcal{M}_G^e(X)$ there exists $x\in X$ and a F\o lner sequence $\mathcal{F}$, such that $x$ is $\mathcal{F}$-generic for $\mu$. 
	See \cite[Theorem 2.4]{fuhrmann2022structure}. 
\end{remark}
	An action $X$ is said to be \emph{uniquely ergodic}, if there exists a unique invariant measure on $X$. 
	An action $X$ is called \emph{pointwise uniquely ergodic}, if for any $x\in X$ there exists a unique invariant measure $\mu_x$ on the orbit closure $\overline{G.x}$. 
	An action $X$ is called \emph{weakly mean equicontinuous}, if it is pointwise uniquely ergodic and the mapping $x\mapsto \mu_x$ is weak*-continuous. Note that weak mean equicontinuous systems are also called \emph{continuously pointwise ergodic} \cite{downarowicz2021all} and can be characterized via a pseudo norm \cite[Theorem 1.6]{xu2022weak}. Since clearly any factor of a weakly mean equicontinuous action is weakly mean equicontinuous we observe each mean equicontinuous action to be weakly mean equicontinuous from \cite[Theorem 1.10]{xu2022weak}. 
	 
	For $(x,x')\in X^2$ we denote $\check{d}(x,x'):=\inf_{g\in G}d(g.x,g.x')$ and $\hat{d}(x,x'):=\sup_{g\in G}d(g.x,g.x')$. 
	A pair $(x,x')\in X^2$ is called \emph{proximal}, if $\check{d}(x,x')=0$. A pair $(x,x')\in X^2$ is called \emph{distal}, if it is not proximal.
	We call a sequence $(x_n,x_n')$ in $X^2$  \emph{asymptotically proximal}, whenever $\check{d}(x_n,x_n')\to 0$. Limit points of asymptotically proximal sequences in $X^2$ are called \emph{regional proximal}. We denote $P$ for the set of all proximal pairs and $RP$ for the set of all regional proximal pairs. 
	An action is called \emph{equicontinuous}, if for any $\epsilon>0$ there exists $\delta>0$ such that for all $x,x'\in X$ we have that $d(x,x')<\delta$ implies $\hat{d}(x,x')<\epsilon$. Recall from \cite{auslander1988minimal} that an action is equicontinuous, if and only if $RP=\Delta_X$. 
	An action is called \emph{distal}, if $P=\Delta_X$. It is called \emph{proximal}, if $P=X^2$. Note that these notions are independent from the choice of a continuous pseudometric. We recommend \cite{auslander1988minimal} for more details. 

\subsection{Factor maps}

	Consider actions of a countable group $G$ on compact metric spaces $X,Y$. 
	A continuous and surjective map $\pi\colon X\to Y$ is called a \emph{factor map}, if $\pi(g.x)=g.\pi(x)$ holds for all $g\in G$ and $x\in X$. 
	If a factor map is injective, it is a homeomorphism, since $X$ and $Y$ are compact metric spaces. In this case we call $\pi$ a \emph{conjugacy}. 
	We denote $\pi_*\colon \mathcal{M}_G(X)\to \mathcal{M}_G(Y)$ for the \emph{push forward} operation, given by $(\pi_*\mu)(f)=\mu(f\circ \pi)$ for $\mu\in \mathcal{M}_G(X)$ and $f\in C(Y)$. 		
	
	Denote $R(\pi):=\{(x,x')\in X^2;\, \pi(x)=\pi(x')\}$. The set of \emph{$\pi$-proximal} pairs is given by $P(\pi):=P\cap R(\pi)$. 
	Limit points of asymptotically proximal sequences in $R(\pi)$ are called \emph{$\pi$-regional proximal} and we denote $RP(\pi)$ for the set of all $\pi$-regionally proximal pairs. Note that $P(\pi)\subseteq RP(\pi)\subseteq R(\pi)$. 

	A factor map $\pi\colon X\to Y$ is called \emph{equicontinuous}, if for all $\epsilon>0$ there exists $\delta>0$ such that $(x,x')\in R(\pi)$ with $d(x,x')<\delta$
	satisfy $\hat{d}(g.x,g.x')<\epsilon$. A factor map $\pi\colon X\to Y$ is equicontinuous, if and only if $RP(\pi)= \Delta_X$ \cite[Chapter 7]{auslander1988minimal}. 		
	A factor map $\pi\colon X\to Y$ is called \emph{distal}, if $P(\pi)= \Delta_X$. 	
	A factor map $\pi\colon X\to Y$ is called \emph{proximal}, if $P(\pi)= R(\pi)$.  
	A factor map $\pi\colon X\to Y$ is called \emph{topo-isomorphic}, if for any $\mu\in \mathcal{M}_G(X)$ there exists $M\subseteq Y$ Borel measurable with $\pi_*\mu(M)=\mu(\pi^{-1}(M))=1$, such that $\pi$ restricts to a bi-measurable bijection $\pi^{-1}(M)\to M$.

\subsection{Besicovitch and Weyl pseudometrics}	
\label{subsec:prelims_besicovitchAndWeyl}

	Consider an action on a compact metrizable space $X$ and a continuous pseudometric $d$ on  $X$. 
For a F\o lner sequence $\mathcal{F}$ we denote 
\[D^\mathcal{F}(x,x'):=\limsup_{n\to \infty}\frac{1}{|F_n|}\sum_{g\in F_n}d(g.x,g.x').\] 
$D^\mathcal{F}$ is called the \emph{Besicovitch pseudometric} w.r.t.\ $\mathcal{F}$ (and $d$). 
Note that the Weyl pseudometric (w.r.t.\ $d$) is then given by 
$D=\sup_\mathcal{F}D^\mathcal{F}$, where the supremum is taken over all F\o lner sequences $\mathcal{F}$ in $G$. 
See \cite{lkacka2018quasi} for further possibilities how to define pseudo metrics that are equivalent to $D^\mathcal{F}$ and $D$.  

	We call a pair $(x,x')\in X^2$ \emph{Banach proximal}, whenever $D(x,x')=0$ and denote $BP$ for the set of Banach proximal pairs. 
	We call a sequence $(x_n,x_n')$ in $X^2$ \emph{asymptotically Banach proximal}, whenever $D(x_n,x_n')\to 0$. Note that since $\check{d}\leq D$ any asymptotically Banach proximal sequence is asymptotically proximal. For a factor map $\pi$ we denote $BP(\pi):=BP\cap R(\pi)$ for the set of all \emph{$\pi$-Banach proximal} pairs. Note that $BP\subseteq P\subseteq RP$ and that $BP(\pi)\subseteq P(\pi)\subseteq RP(\pi)\subseteq R(\pi)$. 
	As already presented in the introduction these notions are independent from the choice of a continuous metric.

\section{The property (M)}\label{sec:about(M)}
	We next characterize the property (M). Recall that mean equicontinuity of an action $X$ is equivalent to (M) for the factor map onto one point. Thus with $R(\pi)=X^2$ the following proposition also yields a characterization of mean equicontinuity for actions. 	

\begin{proposition}\label{pro:M1characterization}
	For a factor map $\pi$ the following statements are equivalent. 
\begin{itemize}
\item[(M)] For $\epsilon>0$ there exists $\delta>0$ such that all $(x,x')\in R(\pi)$ with $d(x,x')<\delta$ satisfy $D(x,x')<\epsilon$.
\item[(i)] Any sequence in $R(\pi)$ that only allows for Banach proximal cluster points is asymptotically Banach proximal. 
\item[(ii)] Any convergent sequence in $R(\pi)$ with a Banach proximal limit is asymptotically Banach proximal. 
\item[(iii)] Any convergent sequence in $R(\pi)$ with a limit in $\Delta_X$ is asymptotically Banach proximal. 
\item[(iv)] A sequence in $R(\pi)$ is asymptotically proximal, if and only if it is asymptotically Banach proximal. 
\end{itemize}		
\end{proposition}
\begin{proof}
'(i)$\Rightarrow$(ii)': Trivial.\\
'(ii)$\Rightarrow$(i)': Let $(x_n,x_n')$ be a sequence in $R(\pi)$ that only allows for Banach proximal cluster points. 
If $(x_n,x_n')$ is not asymptotically Banach proximal, there exists $\delta>0$ and a subsequence $(x_{n_k},x_{n_k}')$ with $D(x_{n_k},x_{n_k}')>\delta$. 
Since $R(\pi)$ is compact, we assume w.l.o.g.\ that $(x_{n_k},x_{n_k}')$ converges to some $(x,x')$. 
Since $(x,x')$ is a cluster point of $(x_n,x_n')$ it is Banach proximal and hence (ii) yields that $(x_{n_k},x_{n_k}')$ must be asymptotically Banach proximal, a contradiction. 

This shows that (i) and (ii) are equivalent. \\
	'(M)$\Rightarrow$(ii)': Consider a convergent sequence $(x_n,x_n')$ in $R(\pi)$ and denote $(x,x')$ for its limit. Let us assume that $(x,x')$ is Banach proximal. 
	Let $\epsilon>0$ and choose $\delta>0$ according to (M). 	
	From $D(x,x')=0$ we observe that there exists $g\in G$ such that $d(g.x,g.x')<\delta$. For sufficiently large $n\in \mathbb{N}$ we observe $d(g.x_n,g.x_n')<\delta$. Since $(g.x_n,g.x_n')\in R(\pi)$, we have $D(x_n,x_n')=D(g.x_n,g.x_n')<\epsilon$ for such large $n$, from which we observe $D(x_n,x_n')\to 0$. \\
	'(ii)$\Rightarrow$(iii)': Trivial, since $\Delta_X\subseteq BP$. \\
	'(iii)$\Rightarrow$(M)': Let us assume that (M) is not satisfied. Then there exists $\epsilon>0$ and a sequence $(x_n,x_n')$ in $R(\pi)$ such that $d(x_n,x_n')\to 0$, while $D(x_n,x_n')\geq \epsilon$. Since $R(\pi)$ is compact we assume without lost of generality that $(x_n,x_n')$ converges to some $(x,x')\in R(\pi)$. From $d(x_n,x_n')\to 0$ we observe  $(x,x')\in \Delta_X$ and by (iii) $(x_n,x_n')$ must be asymptotically Banach proximal, a contradiction. \

This shows that (M), (i) (ii) and (iii) are equivalent. \\
	'(M)$\Rightarrow$(iv)': 
	Since $\check{d}\leq D$ any asymptotically Banach proximal sequence is asymptotically proximal. For the converse consider an asymptotically proximal sequence $(x_n,x_n')$ in $R(\pi)$. There exists a sequence $(g_n)$ in $G$, such that $d(g_n.x_n,g_n.x_n')\to 0$. 
	Let $\epsilon>0$ and choose $\delta>0$ according to (M).
	For large $n$, we observe that $d(g_n.x_n,g_n.x_n')<\delta$ and hence $D(x_n,x_n')=D(g_n.x_n,g_n.x_n')<\epsilon$. This shows $(x_n,x_n')$ to be asymptotically Banach proximal.\\
	'(iv)$\Rightarrow$(iii)': Any sequence $(x_n,x_n')$ with a limit in $\Delta_X$ satisfies $d(x_n,x_n')\to 0$. Since $\check{d}\leq d$, we observe that it must be asymptotically proximal and (iv) yields that it is asymptotically Banach proximal. 
\end{proof}

Recall that by definition a mean equicontinuous factor map satisfies (ii). 

\begin{corollary}
	Any mean equicontinuous factor map satisfies (M). 
\end{corollary}

\begin{remark}
	Note that (ii) can be reformulated as 
\begin{itemize}
\item[(ii')] Every pair in $BP(\pi)$ is a continuity point of 
$D\colon R(\pi)\to [0,\infty)$. 
\end{itemize}
	Furthermore, since $D$ only attains non-negative values, we can reformulate (ii') as the upper-semi continuity of $D\colon R(\pi)\to [0,\infty)$ in every point of $BP(\pi)$. See \cite[A.1.4]{downarowicz2011entropy} for details on upper-semi continuity. 
\end{remark}

\section{Regional Banach proximality}\label{sec:banachProximality}
	
	In analogy with the definition of regional proximality, we call a pair \emph{regionally Banach proximal}, whenever it is the limit of an asymptotically Banach proximal sequence and denote $RBP$ for the set of all regionally Banach proximal pairs. For a factor map $\pi\colon X\to Y$ we call a pair \emph{$\pi$-regionally Banach proximal}, whenever it is the limit of an asymptotically Banach proximal sequence in $R(\pi)$ and denote $RBP(\pi)$ for the set of all $\pi$-regionally Banach proximal pairs. Clearly also $RBP$ and $RBP(\pi)$ do not depend on the choice of a continuous metric on $X$. 
	Note that	
	 $BP\subseteq RBP\subseteq RP$ and that $BP(\pi)\subseteq RBP(\pi)\subseteq RP(\pi)\subseteq R(\pi)$. 
	 
	 \begin{remark}\label{rem:RPisRBPunderM1}
	From the equivalence of (M) and (iv) in Proposition \ref{pro:M1characterization} we observe that any factor map $\pi$ with the property (M) satisfies $RP(\pi)=RBP(\pi)$. 
\end{remark}
	
\begin{proposition}\label{pro:propertiesRBPandBP}
\begin{itemize}
\item[(i)] $BP(\pi)$ is an invariant equivalence relation.
\item[(ii)] $RBP(\pi)$ is closed, invariant, reflexive and symmetric.
\end{itemize}
	
\end{proposition}
\begin{proof}
'(i)': This follows, since $D$ is an invariant pseudometric. \\
'(ii)':	Since $D$ is a pseudometric, we easily observe that $RBP(\pi)$ is reflexive and symmetric. 
	For $(x,x')\in RBP(\pi)$ we find an asymptotically Banach proximal sequence $(x_n,x_n')$ in $R(\pi)$ with $(x_n,x_n')\to (x,x')$. For $g\in G$, we clearly have $(g.x_n,g.x_n')\to (g.x,g.x')$. Since $D$ is invariant, $(g.x_n,g.x_n')$ is also asymptotically Banach proximal and we observe $(g.x,g.x')\in RBP(\pi)$. This shows $RBP(\pi)$ to be invariant. 
	
	Let now $(x,x')\in \overline{RBP(\pi)}$ and denote $d'$ for a continuous pseudometric on $X^2$. 
	For $n\in \mathbb{N}$ there exists $(y_n,y_n')\in RBP(\pi)$ with $d'((x,x'),(y_n,y_n'))<1/n$. 	
	Since $(y_n,y_n')\in RBP(\pi)$ there exists $(x_n,x_n')\in R(\pi)$, such that $D(x_n,x_n')<1/n$ and $d'((x_n,x_n'),(y_n,y_n'))<1/n$. We thus observe that $(x_n,x_n')$ is an asymptotically Banach proximal sequence in $R(\pi)$ that converges to $(x,x')$. Thus $(x,x')\in RBP(\pi)$, which shows $RBP(\pi)$ to be closed.  
\end{proof}

\begin{remark}
	Note that any action allows for a maximal mean equicontinuous factor \cite[Theorem 3.10.]{li2015mean}. 
	For an action the smallest invariant closed equivalence relation (\emph{icer}) that contains $RP$ defines it's maximal equicontinuous factor \cite{auslander1988minimal}.	
	It is natural to ask, whether the smallest icer that contains $RBP$ defines the maximal mean equicontinuous factor. This is not the case.
	 
	It was shown in \cite[Theorem 3.5.]{li2015mean} that for mean equicontinuous actions, we have $RP=BP$. 
	Since $BP\subseteq RBP\subseteq RP$, we observe that $BP=RBP=RP$ to be an icer. 
	Thus for mean equicontinuous actions the smallest icer that contains $RBP$ gives the maximal equicontinuous factor and not necessarily the action itself. See \cite[Chapter 6]{qiu2020note} for a characterization of the relation of the maximal mean equicontinuous factor in the context of $\mathbb{Z}$ actions. 
\end{remark}

\section{Mean equicontinuity and (M)}

	In Example \ref{exa:M2doesNotImplyM1} below we will see that (M) is not enough to ensure that $BP(\pi)=RP(\pi)$. This motivated us to use a stronger condition than (M) for our definition of mean equicontinuity for factor maps. 

\begin{proposition}\label{pro:characterizationM2}
	Consider a factor map $\pi$. 
The following statements are equivalent. 
\begin{itemize}
\item[(i)] $\pi$ is mean equicontinuous. 
\item[(ii)] $\pi$ satisfies (M) and $RBP(\pi)=BP(\pi)$. 
\item[(iii)] $\pi$ satisfies (M) and $RP(\pi)=BP(\pi)$. 
\end{itemize}
\end{proposition}
\begin{proof}
From Remark \ref{rem:RPisRBPunderM1} we observe the equivalence of (ii) and (iii). \\
'(i)$\Rightarrow$(ii)': Recall that $BP(\pi)\subseteq RBP(\pi)$ is always satisfied. For the converse let $(x,x')\in RBP(\pi)$ and consider an asymptotically Banach proximal sequence $(x_n,x_n')$ in $R(\pi)$ that converges to $(x,x')$. 
	Since $\pi$ is assumed to be mean equicontinuous we observe $(x,x')\in BP$. Since $R(\pi)$ is closed we have $(x,x')\in BP(\pi)$. \\
'(ii)$\Rightarrow$(i)': Consider a convergent sequence $(x_n,x_n')$ in $R(\pi)$ and denote $(x,x')$ for it's limit. 
	Since $\pi$ is assumed to satisfy (M) we observe from Proposition \ref{pro:M1characterization} that whenever $(x,x')$ is Banach proximal, then $(x_n,x_n')$ is asymptotically Banach poximal. For the converse assume that $(x_n,x_n')$ is asymptotically Banach poximal. Thus $(x,x')\in RBP(\pi)=BP(\pi)$ and hence $(x,x')$ is Banach proximal.  
\end{proof}

In \cite[Theorem 3.5]{li2015mean} and \cite[Chapter 3]{fuhrmann2022structure} it was shown that for mean equicontinuous actions we have $BP=P=RP$. We next show that with our methods we can achieve a generalization of this result to the relative context.
 
 \begin{theorem}\label{the:MEimpliesRPisRBPisPisBPisICER}
For any mean equicontinuous factor map $\pi$ we have that 
\[RP(\pi)=RBP(\pi)=P(\pi)=BP(\pi)\] is a closed invariant equivalence relation. 
\end{theorem}
\begin{proof}
	From Proposition \ref{pro:characterizationM2} we observe that $BP(\pi)=RBP(\pi)=RP(\pi)$. Thus Proposition \ref{pro:propertiesRBPandBP} yields that this relation is a closed invariant equivalence relation. Since $\check{d}\leq D$, we have $BP(\pi)\subseteq P(\pi)\subseteq RP(\pi)$. 
\end{proof}

	It is well known that images of (regional) proximal pairs under factor maps are (regional) proximal. A similar statement also holds in the context of Banach proximality as we will see next.  

\begin{proposition}\label{pro:pushingDownBPandaBP}
	Consider a factor map $\pi\colon X\to Y$. 
	If $(x,x')\in X^2$ is Banach proximal, then also $(\pi(x),\pi(x'))$ is Banach proximal. If $(x_n,x_n')$ is an  asymptotically Banach proximal sequence in $X^2$, then also $(\pi(x_n),\pi(x_n'))$ is asymptotically Banach proximal. If $(x,x')$ is regionally Banach proximal, then also $(\pi(x),\pi(x'))$ is regionally Banach proximal. 
\end{proposition}
\begin{proof}
Denote $d_X$ and $d_Y$ for continuous metrics on $X$ and $Y$, respectively. 
	Clearly $d_X+d_Y\circ\hat{\pi}$ is a continuous metric as the sum of a continuous metric and a continuous pseudometric.
	We thus assume w.l.o.g.\ that $d_Y\circ\hat{\pi}\leq d_X$. We observe that the corresponding Weyl pseudometrics also satisfy $D_Y\circ \hat{\pi}\leq D_X$. 
	From this observation the statement follows easily. 
\end{proof}

\section{Examples}\label{sec:examples}

The following example shows that there exist equicontinuous factor maps with $D\notin C(R(\pi))$. 

\begin{example}\label{exa:Dnotcontinuous}
	Let $Y:=[0,1]$. For $y\in Y$ we denote $\hat{y}:=(y,y)$ and $\check{y}:=(y,-y)$. Equip 
	$X:=\{\hat{y};\, y\in Y\}\cup \{\check{y};\, y\in Y\}$ with the euclidean metric of $\mathbb{R}^2$ and denote $\pi\colon X\to Y$ for the projection $(y,\pm y)\mapsto y$. 
		
	Now consider a homeomorphism $S\colon Y\to Y$ that fixes exactly the points of $N:=\{\frac{1}{n};\, n\in \mathbb{N}\}\cup \{0\}$ and moves all other points to the left, i.e.\ which satisfies $Sy<y$ for all $y\in Y\setminus N$.
	On $X$, we consider the action $T(y,\pm y):=(Sy,\pm Sy)$ and observe $\pi$ to be a factor map. 
	Consider $y\in Y\setminus N$. We observe that there is $n\in \mathbb{N}$, with $y\in (\frac{1}{n+1},\frac{1}{n})$. Clearly for all $k\in \mathbb{Z}$, we have $d(T^k\hat{y},T^k\check{y})\in (\frac{2}{n+1},\frac{2}{n})$. It is straight forward to observe that $\hat{d}(\hat{y},\check{y})=\frac{2}{n}$. This shows that that for any $m\in \mathbb{N}$ all $(x,x')\in R(\pi)$ with $d(x,x')<\frac{2}{m}$ satisfy $\hat{d}(x,x')\leq \frac{2}{m}$, which shows $\pi$ to be an equicontinuous extension. 	
	Furthermore, considering the F\o lner sequence $\mathcal{F}=(F_k)$ with $F_k:=\{-k,\dots,0\}$, we observe $D(\check{y},\hat{y})=\frac{2}{n}$ for $y\in (\frac{1}{n+1},\frac{1}{n})$. From here it is straight forward to see that both $D$ is not continuous.
\end{example}

	The following example gives a factor map $\pi$ that satisfies (M), but which is not mean equicontinuous. It satisfies 
	$BP(\pi)=P(\pi)\subsetneq RBP(\pi)=RP(\pi)=R(\pi)$. 
	
\begin{example}\label{exa:M2doesNotImplyM1}
	We next consider $X$ as in \cite[Example 5.3]{downarowicz2021all}. We include a short introduction of this example for the convenience of the reader.  
	Consider $K:=\{z\in \mathbb{R}^2;\, d((0,1),z)=1\}$  and note that $(0,0)\in K$. Denote 
	$C_k:=K\times \{1/k\}$ and $C:=K\times \{0\}$. Consider $X:=C\cup \bigcup_{k\geq 1}C_k$ with the topology induced by $\mathbb{R}^3$. Denote $c_k:=(0,0,1/k)$ and $c:=(0,0,0)$. 
	Consider an action of $\mathbb{Z}$ on $X$ that leaves each $C_k$ invariant, fixes all $c_k$ as well as $c$ and moves all other points clockwise, but with the speed of the movement decaying as $k$ increases, so that on $C$ we have the identity map. 
	
	Consider the factor map $\pi\colon X\to Y:=\{0\}\cup \{1/k;\, k\in \mathbb{N}\}$, where we impose on $Y$ the subspace topology of $\mathbb{R}$ and act trivially on $Y$.
	Denote $d$ for the metric on $X$. 
	Since $C=\pi^{-1}(0)$ is fixed, we observe that $D=d$ on $C$. Furthermore, it is straight forward to observe that $D\equiv 0$ on each of the $C_k$. 
	From Proposition \ref{pro:M1characterization} it is easy to observe that $\pi$ satisfies (M). 
	Clearly $BP(\pi)=\bigcup_{k\in \mathbb{N}} C_k^2$ and $RBP(\pi)=R(\pi)$. Thus by Proposition \ref{pro:characterizationM2} $\pi$ is not mean equicontinuous.
	Note that $RBP(\pi)\subseteq RP(\pi)\subseteq R(\pi)$ yields that $RBP(\pi)=RP(\pi)$.  
	Furthermore, note that $BP(\pi)=P(\pi)$ in this example. 
\end{example}

\section{Some useful tools concerning topo-isomorphism}
\label{sec:toolsTopoIsom}

	Recall that the \emph{(closed) unit ball} of $C(X)$ is given by $\{f\in C(X);\, \|f\|_\infty\leq 1\}$. 
	
\begin{lemma}
\label{lem:functionApproximationTopoIsomorphy}
	Let $\pi\colon X\to Y$ be a topo-isomorphic factor map and $f$ in the unit ball of $C(X)$. For any $\epsilon>0$ and any $\mu\in \mathcal{M}_G(X)$ there exists $h$ in the unit ball of $C(Y)$ such that $\|f-h\circ\pi\|_{L^1(\mu)}<\epsilon$. 
\end{lemma}
\begin{proof}
	Note first that $C(Y)$ is dense in $L^1(\pi^*\mu)$. Since $\pi$ is a topo-isomorphy, we observe from the arguments presented in \cite[Chapter 2]{walters2000introduction} that $\pi^*\colon L^1(\pi^*\mu)\to L^1(\mu)$ with $h\mapsto h\circ\pi$ is an isometric isomorphism. Thus $\{h\circ \pi;\, h\in C(Y)\}$ is densely mapped into $L^1(\mu)$ and we observe that there exists $h'\in C(Y)$ with $\|f-h'\circ\pi\|_{L^1(\mu)}<\epsilon$. 
	
	Consider $M:=\{y\in Y;\, |h'(y)|> 1\}$ and define $h\colon Y\to \mathbb{C}$ by 
	$h(y):=h'(y)/|h'(y)|$ for $y\in M$ and $h(y):=h'(y)$ everywhere else. 
	We observe that $\|h\|_\infty\leq 1$. 
	Furthermore, since for any $z,z'\in \mathbb{C}$ with $|z|\leq 1$ and $|z'|>1$ we have that $|z-z'/|z'||\leq |z-z'|$ we observe that 
	$|f(x)-h\circ \pi(x)|\leq |f(x)-h'\circ \pi(x)|$ for all $x\in X$ and hence 
	$\|f-h\circ\pi\|_{L^1(\mu)}\leq \|f-h'\circ\pi\|_{L^1(\mu)}<\epsilon$. 
\end{proof}

	Recall that a sequence $(f_n)$ in $C(X)$ is said to \emph{separate points}, whenever for $x,x'\in X$ there exists $n$, such that $f_n(x)\neq f_n(x')$. 

\begin{lemma}\label{lem:NiceFamilyOfFunctions}
		Let $\pi\colon X\to Y$ be a topo-isomorphic factor map. There exists a point separating sequence $(f_m)_{m\in \mathbb{N}}$ in the unit ball of $C(X)$ such that for any $\epsilon>0$ and any $\mu\in \mathcal{M}_G(X)$ there exists a point separating sequence $(h_m)_{m\in \mathbb{N}}$ in the unit ball of $C(Y)$ with $\sup_m \|f_m-h_m\circ \pi\|_{L^1(\mu)}\leq \epsilon$. 
	\end{lemma}
	\begin{proof}
	We start with a countable subsets $A\subseteq C(X)$ and $B\subseteq C(Y)$ that separate points respectively and set $C:=A\cup \{h\circ \pi;\, h\in B\}$. W.l.o.g.\ we assume $\|h\|_\infty \leq 1$ and $\|f\|_\infty\leq 1$ for all $f\in A$ and $h\in B$.
	We enumerate $C$ to get a sequence $(f_m)_{m\in \mathbb{N}}$ in $C(X)$. 
	
	Now let $\epsilon>0$ and $\mu\in \mathcal{M}_G(X)$. Consider $m\in \mathbb{N}$. If $f_m\in C\setminus A$, then it is of the form $h_m\circ \pi$ for some $h_m\in B$ and we trivially have $\|f_m-h_m\circ \pi\|_{L^1(\mu)}=0<\epsilon$. 
	If $f_m\in A$, then Lemma \ref{lem:functionApproximationTopoIsomorphy} allows to choose $h_m\in C(Y)$, such that $\|h_m\|_\infty\leq 1$ and 
$\|f_m-h_m\circ \pi\|_{L^1(\mu)}\leq \epsilon$.
	Since $B\subseteq \{h_m;\, m\in \mathbb{N}\}$ we observe that $(h_m)_{m\in \mathbb{N}}$ separates points in $Y$. 
	\end{proof}

	For a sequence $\mathfrak{f}=(f_m)_{m\in \mathbb{N}}$ in the unit ball of $C(X)$ with we define a continuous pseudometric by 
$d_\mathfrak{f}(x,x'):=\sum_{m=1}^\infty 2^{-m}|f_m(x)-f_m(x')|. $
For a F\o lner sequence $\mathcal{F}$ we denote $D_\mathfrak{f}^\mathcal{F}$ for the respective Besicovitch pseudometric and $D_\mathfrak{f}$ for the respective Weyl pseudometric.

\begin{lemma}\label{lem:controlWeylPMUnderFactorMap}
Let $X$ be an action and $\mathcal{F}$ be a F\o lner sequence in $G$. Let $x_i\in X$
	be $\mathcal{F}$-generic for $\mu_i\in \mathcal{M}_G(X)$. Denote $\mu:=\frac{1}{2}(\mu_1+\mu_2)$.
For any sequences $\mathfrak{f}=(f_m)$ and
	$\mathfrak{h}=(h_m)$ in the unit ball of $C(X)$ we have
	\[D_\mathfrak{f}^\mathcal{F}(x_1,x_2)\leq D_\mathfrak{h}^\mathcal{F}(x_1,x_2)+ 2\sum_{m=1}^\infty 2^{-m}\|f_m-h_m\|_{L^1(\mu)}.\] 
\end{lemma}	
\begin{proof}
	Note that $\|f_m-h_m\|_{L^1(\mu)}\leq \|f_m-h_m\|_\infty\leq 2$ and hence that $f:=\sum_{m=1}^\infty 2^{-m}|f_m-h_m|\in C(X)$ and $2\sum_{m=1}^\infty 2^{-m}\|f_m-h_m\|_{L^1(\mu)}<\infty$. Denote $\pi^{(i)}\colon X^2\to X$ for the respective projections and $\phi:=\sum_{i=1}^2 f\circ \pi^{(i)}$.	
	For $g\in G$ we have $\phi(g.x_1,g.x_2)=\sum_{i=1}^2 f(g.x_i)=\sum_{i=1}^2 g.\delta_{x_i}(f)$. 
Since $x_i$ is $\mathcal{F}$-generic for $\mu_i$ we observe 	
\begin{align}\label{equation:ControlPhi}
\frac{1}{|F_n|}\sum_{g\in F_n}\phi(g.x_1,g.x_2)=\sum_{i=1}^2\frac{1}{|F_n|}\sum_{g\in F_n}g.\delta_{x_i}\left(f\right)\to \sum_{i=1}^2\mu_i(f).
\end{align}	
Denote $d_m(x,x'):=|f_m(x)-f_m(x')|$ and $d_m'(x,x'):=|h_m(x)-h_m(x')|$ for $x,x'\in X$. A straight forward computation reveals that we have $d_m\leq d_m'+\sum_{i=1}^2 |f_m-h_m|\circ \pi^{(i)}$. 
Note that 
\[\sum_{m=1}^\infty 2^{-m} \sum_{i=1}^2 |f_m-h_m|\circ \pi^{(i)}=\sum_{i=1}^2  f\circ \pi^{(i)}=\phi.\] 
	This allows to compute 
	$d_\mathfrak{f}
= \sum_{m=1}^\infty 2^{-m}d_m
\leq \sum_{m=1}^\infty 2^{-m} d_m' + \phi
=d_\mathfrak{h}+\phi$.
	From (\ref{equation:ControlPhi}) we thus observe
\begin{align*}
D_\mathfrak{f}^\mathcal{F}(x_1,x_2)
&=\limsup_{n\to \infty}\frac{1}{|F_n|}\sum_{g\in F_n}d_\mathfrak{f}(g.x_1,g.x_2)\\
&\leq\limsup_{n\to \infty}\frac{1}{|F_n|}\sum_{g\in F_n}d_\mathfrak{h}(g.x_1,g.x_2)+\lim_{n\to \infty}\frac{1}{|F_n|}\sum_{g\in F_n}\phi(g.x_1,g.x_2)\\
&= D_\mathfrak{h}^\mathcal{F}(x_1,x_2)
	+ \sum_{i=1}^2\mu_i(f).
\end{align*}
From
\[\mu_1(f)+\mu_2(f)\leq 2\mu(|f|)=2\|f\|_{L^1(\mu)}\leq 2\sum_{m=1}^\infty 2^{-m}\|f_m-h_m\|_{L^1(\mu)}\]
we thus observe 
$D_\mathfrak{f}^\mathcal{F}(x_1,x_2)\leq D_\mathfrak{h}^\mathcal{F}(x_1,x_2)+2\sum_{m=1}^\infty 2^{-m}\|f_m-h_m\|_{L^1(\mu)}$.
\end{proof}

\section{Characterizations of topo-isomorphy}
\label{sec:characterizationsTopoIsomorphism}

	Recall that we define a factor map $\pi\colon X\to Y$ to be \emph{Banach proximal}, whenever all pairs $(x,x')\in R(\pi)$ are Banach proximal. Clearly Banach proximal factor maps are mean equicontinuous. 
	From \cite{fuhrmann2022structure, qiu2020note} we know that a factor map $\pi$ onto an equicontinuous action $Y$ is topo-isomorphic, if and only if it is Banach proximal. We have the following.  

\begin{theorem}\label{the:topo_iso_and_Banach_prox}
	For any factor map $\pi\colon X\to Y$ the following statements are equivalent. 
\begin{itemize}
\item[(i)] $\pi$ is topo-isomorphic.
\item[(ii)] $\pi$ is Banach proximal.
\item[(iii)] For any $(x,x')\in X^2$ we have that $(x,x')$ is Banach proximal, whenever $(\pi(x),\pi(x'))$ is Banach proximal. 
\end{itemize}
\end{theorem}

\begin{corollary}\label{cor:topo_iso_is_mean_eq}
	Any topo-isomorphic factor map is mean equicontinuous. 
\end{corollary}

\begin{remark}
	Note that the argument for '(ii)$\Rightarrow$(i)' is the argument from \cite[Theorem 3.8]{li2015mean} or \cite[Theorem 3.15]{fuhrmann2022structure}, which carries over to the general setting without problems. We include the argument for the convenience of the reader. 
\end{remark}

\begin{proof}
'(i)$\Rightarrow$(iii)':  Consider a point separating sequence $\mathfrak{f}=(f_m)$ in the unit ball of $C(X)$ as constructed in  Lemma \ref{lem:NiceFamilyOfFunctions} and the continuous metric $d_\mathfrak{f}$ on $X$. 
	Let $\epsilon>0$. 
	Let $\mathcal{F}$ be a F\o lner sequence in $G$ such that $D_\mathfrak{f}^\mathcal{F}(x,x')+\epsilon\geq D_\mathfrak{f}(x,x')$. 
	By considering a subsequence and by a standard Krylov–Bogolyubov argument we can assume w.l.o.g.\ that 
	$x$ and $x'$ are $\mathcal{F}$-generic for some $\mu,\mu'\in \mathcal{M}_G(X)$, respectively. 
	We denote $\nu:=1/2(\mu+\mu')$. 
Our choice of $\mathfrak{f}$ allows to choose a point separating sequence $\mathfrak{h}=(h_m)$ in the unit ball of $C(Y)$ that satisfies 
$\sup_m\|f_m-h_m\circ \pi\|_{L^1(\nu)}\leq \epsilon$.
Denote $\mathfrak{h}':=(h_m\circ \pi)$. 
	Since $(\pi(x),\pi(x'))$ is Banach proximal and $d_\mathfrak{h}$ is a continuous metric on $Y$ we have $D_{\mathfrak{h}'}^\mathcal{F}(x,x')=D_\mathfrak{h}^\mathcal{F}(\pi(x),\pi(x'))=0$. 	
	By Lemma \ref{lem:controlWeylPMUnderFactorMap} we observe that 
\begin{align*}
	D_\mathfrak{f}^\mathcal{F}(x,x')&\leq D_{\mathfrak{h}'}^\mathcal{F}(x,x')+ 2\sum_{m=1}^\infty 2^{-m}\|f_m-h_m\circ \pi\|_{L^1(\nu)}
	\leq 0 + 2\epsilon.
\end{align*}
	We thus have $D_\mathfrak{f}(x,x')\leq 3\epsilon$ and since $\epsilon>0$ was arbitrary, we have $D_\mathfrak{f}(x,x')=0$. Since $d_\mathfrak{f}$ is a continuous metric on $X$ this shows $(x,x')$ to be Banach proximal. 	\\
'(iii)$\Rightarrow$(ii)': Trivial. \\
'(ii)$\Rightarrow$(i)': Assume that $\pi$ is Banach proximal and consider $\mu\in \mathcal{M}_G(X)$. Denote $\nu:=\pi^*\mu$. Let $\mu=\int_Y \mu_yd\nu(y)$ be the desintegration of $\mu$ over $\nu$. And consider the relative product measure  
	$\mu \times_\nu \mu:=\int_Y \mu_y\times \mu_y d\nu(y)$ on $R(\pi)$. See \cite{glasner2003ergodic} for details on these notions. $G$ acts on $R(\pi)$ by $g.(x,x'):=(g.x,g.x')$ and clearly 
	$\mu \times_\nu \mu \in \mathcal{M}_G(R(\pi))$. 
	If $\mu \times_\nu \mu$ is not supported on $\Delta_X$, there exists an open set $U\subseteq R(\pi)$ with positive distance from $\Delta_X$ and 
	satisfies $\mu \times_\nu \mu(U)>0$. 	
	By considering the ergodic decomposition of $\mu \times_\nu \mu$ there exists $\kappa\in \mathcal{M}_G^e(R(\pi))$ such that 
	$\kappa(U)>0$. 
	By Remark \ref{rem:genericityCreation} there exists $(x,x')\in R(\pi)$ and a F\o lner sequence $\mathcal{F}=(F_n)_{n\in \mathbb{N}}$ such that $(x,x')$ is $\mathcal{F}$-generic for $\kappa$ and we conclude $D(x,x')\geq \lim_{n\to \infty} 1/|F_n|\sum_{g\in F_n} d(g.(x,x'))=\int d d\kappa>0$, a contradiction to the Banach proximality of $\pi$. 
	
	We have shown that the support of $\mu \times_\nu \mu$ is contained inside $\Delta_X$. Thus $\check{Y}:=\{y\in Y;\,\exists x\in X\colon \mu_y=\delta_x\}$ satisfies $\nu(\check{Y})=1$. 
	Denote $\check{X}$ for the set of all $x\in X$ with $\{x\}$ being the support of some $\mu_y$ for $y\in Y$. Clearly $\mu(\check{X})=1$ and the restriction of $\pi\colon \check{X}\to \check{Y}$ is bijective. 
\end{proof}

	For factor maps between weakly mean equicontinuous actions we have the following.

\begin{theorem}\label{the:charTOPOISOMWeakMeanEqui}
For a factor map $\pi\colon X\to Y$ between weakly mean equicontinuous actions the following statements are equivalent.
\begin{itemize}
\item[(i)] $\pi$ is topo-isomorphic. 
\item[(ii)] For any convergent sequence $(x_n,x_n')$ in $X^2$ we have that $(x_n,x_n')$ is an asymptotically  Banach proximal sequence, whenever $(\pi(x_n),\pi(x_n'))$ is an asymptotically Banach proximal sequence. 
\item[(iii)] For any sequence $(x_n,x_n')$ in $X^2$ we have that $(x_n,x_n')$ is an asymptotically Banach proximal sequence, whenever $(\pi(x_n),\pi(x_n'))$ is an asymptotically Banach proximal sequence. 
\end{itemize}
\end{theorem}
\begin{proof}
'(i)$\Rightarrow$(ii)':
Consider a point separating sequence $\mathfrak{f}=(f_m)$ in the unit ball of $C(X)$ as constructed in  Lemma \ref{lem:NiceFamilyOfFunctions} and recall that $d_\mathfrak{f}$ is a continuous metric on $X$. It thus suffices o show that $D_\mathfrak{f}(x_n,x_n')\to 0$.  

Let $\epsilon>0$ and denote $(x,x')$ for the limit of $(x_n,x_n')$ and $y:=\pi(x), y':=\pi(x'), y_n:=\pi(x_n)$, as well as $y_n':=\pi(x_n')$. Denote $\mu,\mu',\mu_n$ and $\mu_n'$ for the unique invariant Borel probability measures on the orbit closures of $x,x',x_n$ and $x_n'$, respectively. Denote $\nu_:=\frac{1}{2}(\mu+\mu')$ and $\nu_n:=\frac{1}{2}(\mu_n+\mu_n')$. Note that since $X$ is assumed to be weak mean equicontinuous, we have that $\nu_n\to \nu$. 	
	Consider a point separating	sequence $\mathfrak{h}=(h_m)$ in the unit ball of $C(Y)$ with $\sup_m\|f_m-h_m\circ \pi\|_{L^1(\nu)}\leq \epsilon$ and recall that $d_\mathfrak{h}$ is a continuous metric on $Y$.
	We have 
	$\|f_m-h_m\circ \pi\|_{L^1(\nu_n)}\leq \|f_m-h_m\circ \pi\|_\infty\leq 2$ for all $m,n\in \mathbb{N}$. Since $\nu_n\to \nu$ we observe that there exists $N\in \mathbb{N}$ such that 
\begin{align*}
2\sum_{m=1}^\infty 2^{-m}\|f_m-h_m\circ \pi\|_{L^1(\nu_n)} \leq 3\epsilon
\end{align*}
for all $n\geq N$. 
Furthermore, since $(y_n,y_n')$ is assumed to be asymptotically Banach proximal, we can choose $N$ large enough such that $D_\mathfrak{h}(y_n,y_n')\leq \epsilon$ holds for all $n\geq N$. 	

	Let $\mathcal{F}$ be a F\o lner sequence in $G$. 
	Denoting $\mathfrak{h}':=(h_m\circ \pi)$, we observe 
	$D_\mathfrak{h'}^\mathcal{F}(x_n,x_n')= D_\mathfrak{h}^\mathcal{F}(y_n,y_n')\leq D_\mathfrak{h}(y_n,y_n')\leq \epsilon$. 	
	Note that by the pointwise unique ergodicity of $X$ we have that 
	$x_n$ and $x_n'$ are $\mathcal{F}$-generic for $\mu_n$ and $\mu_n'$, respectively. 	
	From Lemma \ref{lem:controlWeylPMUnderFactorMap}
	we observe that for $n\geq N$ we have
	\begin{align*}	
	D_\mathfrak{f}^\mathcal{F}(x_n,x_n')
	&\leq D_\mathfrak{h'}^\mathcal{F}(x_n,x_n')+ 2\sum_{m=1}^\infty 2^{-m}\|f_m-h_m\circ \pi\|_{L^1(\nu_n)}
	\leq 4\epsilon.
	\end{align*}
	Since the same $N$ works for all F\o lner sequences $\mathcal{F}$ we observe $D_\mathfrak{f}(x_n,x_n')\leq 4\epsilon$ for all $n\geq N$. This shows that $(x_n,x_n')$ is an asymptotically Banach proximal sequence.  \\
'(ii)$\Rightarrow$(iii)':
If $(x_n,x_n')$ is not asymptotically Banach proximal, there exists a subsequence $(x_{n_k},x_{n_k}')$ and $\delta>0$ such that $D_X(x_{n_k},x_{n_k}')\geq \delta$ for all $k$.
	Since $X^2$ is compact we assume w.l.o.g.\ that $(x_{n_k},x_{n_k}')$ is convergent. Since clearly $(\pi(x_{n_k}),\pi(x_{n_k}'))$ is still asymptotically Banach proximal (ii) yields that $(x_{n_k},x_{n_k}')$ must be asymptotically Banach proximal, a contradiction. \\
(iii)$\Rightarrow$(i): Consider $(x,x')\in R(\pi)$ and observe that trivially $\pi(x)=\pi(x')$. Thus trivially $(\pi(x),\pi(x'))_{n\in \mathbb{N}}$ is asymptotically Banach proximal. 
From (iii) we observe that $(x,x')_{n\in \mathbb{N}}$ is asymptotically Banach proximal and hence that $(x,x')$ is Banach proximal. This shows that $\pi$ is a Banach proximal factor map and we observe (i) from Theorem \ref{the:topo_iso_and_Banach_prox}. 
\end{proof}

\section{Mean equicontinuity, distality and proximality}
\label{sec:meanEquicontinuityDistalityEquicontinuity}

\subsection{Mean equicontinuity and distality}

Recall that a factor map $\pi$ is called distal, whenever $P(\pi)=\Delta_X$.
	A factor map $\pi\colon X\to Y$ is called \emph{Banach distal}, whenever $BP(\pi)=\Delta_X$, i.e.\ whenever for distinct $x,x'\in X$ with $\pi(x)=\pi(x')$ we have $D(x,x')>0$. Since $BP(\pi)\subseteq P(\pi)$, we observe that any distal factor map is Banach distal. In Example \ref{exa:topo-isometricAfterEquicont} below, we will see that there exist Banach distal factor maps, which are not distal. 
	
	Recall that any equicontinuous action is mean equicontinuous and distal. 
	In \cite[Corollary 3.6]{li2015mean}	it is presented that mean equicontinuous and distal actions are equicontinuous. The following theorem is a generalization of this result to factor maps. 

\begin{theorem}\label{the:char_equicontinuity}
	For a  factor map $\pi$ the following statements are equivalent.
	\begin{itemize}
	\item[(i)] $\pi$ is equicontinuous.
	\item[(ii)] $\pi$ is mean equicontinuous and distal. 
	\item[(iii)] $\pi$ is mean equicontinuous and Banach distal.
	\end{itemize}
\end{theorem}
\begin{proof}
'(i)$\Rightarrow$(ii)': It is well known that any equicontinuous factor map is distal \cite{auslander1988minimal}. Since $D\leq \hat{d}$ any equicontinupus factor map satisfies (M). Furthermore, an equicontinuous factor map $\pi$ satisfies $RP(\pi)=\Delta_X$ \cite{auslander1988minimal}. Hence $\Delta_X\subseteq BP(\pi)\subseteq RBP(\pi)\subseteq RP(\pi)=\Delta_X$ and we observe $BP(\pi)=RP(\pi)$. Thus Proposition \ref{pro:characterizationM2} yields that $\pi$ is mean equicontinuous. \\
'(ii)$\Rightarrow$(iii)': Trivial, since any distal factor map is Banach distal. \\
'(ii)$\Rightarrow$(iii)': From Proposition \ref{pro:characterizationM2} and the Banach distality of $\pi$, we observe that $RP(\pi)=BP(\pi)=\Delta_X$. Thus $\pi$ is equicontinuous. 
\end{proof}

\subsection{Mean equicontinuity and proximality}

Since $BP(\pi)\subseteq P(\pi)$ any Banach proximal factor map is proximal. Since trivially any Banach proximal factor map is mean equicontinuous, we observe that any Banach proximal is proximal and mean equicontinuous. 	
	In the proof of \cite[Theorem 3.5]{li2015mean} it is shown (for $G=\mathbb{Z}$ and with different techniques) that for a mean equicontinuous action the proximal relation and the Banach proximal relation coincide. Thus this result can be interpreted as follows. If $Y$ is the one point system, then a factor map to $Y$ is Banach proximal, if and only if it is mean equicontinuous and proximal. Since topo-isomorphy and Banach proximality are equivalent notions (see Theorem \ref{the:topo_iso_and_Banach_prox}) we have the following generalization. 
	
\begin{theorem}\label{the:characterization_Banach_proximal_factor_maps}
	For a factor map $\pi\colon X\to Y$ the following statements are equivalent.
	\begin{itemize}
	\item[(i)] $\pi$ is topo-isomorphic.
	\item[(ii)] $\pi$ is mean equicontinuous and proximal. 
	\item[(iii)] $\pi$ is mean equicontinuous and $R(\pi)=RP(\pi)$. 
	\end{itemize}
\end{theorem}
\begin{proof}
	Recall from Theorem \ref{the:topo_iso_and_Banach_prox} that topo-isomorphy and Banach proximality are equivalent notions for factor maps. Recall that any Banach proximal factor map is trivially mean equicontinuous. From $BP(\pi)\subseteq P(\pi)\subseteq RP(\pi)$ we thus observe that (i)$\Rightarrow$(ii)$\Rightarrow$(iii). 
	If $\pi$ is mean equicontinuous and satisfies $R(\pi)=RP(\pi)$, then by Theorem \ref{the:MEimpliesRPisRBPisPisBPisICER} we observe $BP(\pi)=RP(\pi)=R(\pi)$ and hence $\pi$ is Banach proximal. 
\end{proof}

It is well known that the only topo-isomorphic factor maps between equicontinuous actions are conjugacies. For reference see \cite{feres2002ergodic} and \cite[Corollary 3.17]{fuhrmann2022structure}. Theorem \ref{the:characterization_Banach_proximal_factor_maps} allows to observe this rigidity also for distal actions. 
	
\begin{corollary}\label{cor:topoIsomBetweenDistalSystemsIsCOnj}
	Any topo-isomorphic factor map $\pi$ between distal actions is a conjugation. 
\end{corollary}
\begin{proof}
	By Theorem \ref{the:characterization_Banach_proximal_factor_maps} any topo-isomorphic factor map is proximal. Since any factor map between distal actions is distal, we observe $\pi$ to be a conjugacy. 
\end{proof}

\section{Decomposing mean equicontinuous factor maps}\label{sec:decomposition}

	It is well known that any mean equicontinuous system is the topo-isomorphic extension of an equicontinuous system (it's maximal equicontinuous factor) \cite{downarowicz2016isomorphic, fuhrmann2022structure, li2015mean}. This can be interpreted as the statement that any mean equicontinuous factor map onto one point is the composition of a topo-isomorphy and an equicontinuous factor map. We will see that this composition holds true for general factor maps, whenever $X$ is minimal or weak mean equicontinuous. It remains open, whether such a decomposition is possible in general.

\subsection{Decomposition for minimal actions}

\begin{theorem}\label{the:decompositionMinimal}
	Any mean equicontinuous factor map $\pi\colon X\to Y$ between minimal actions decomposes as $\pi=\psi\circ \phi$ into a topo-isomorphic factor map $\phi\colon X\to X\big/BP(\pi)$ and an equicontinuous factor map $\psi\colon X\big/BP(\pi)\to Y$.
\end{theorem}
\begin{proof}
	In Theorem \ref{the:MEimpliesRPisRBPisPisBPisICER} we have already seen that $BP(\pi)$ is a closed invariant equivalence relation. This defines a factor map $\phi\colon X\to X\big/BP(\pi)$ with $R(\phi)=BP(\pi)$. Clearly $\phi$ is Banach proximal and by Theorem \ref{the:topo_iso_and_Banach_prox} we observe that $\phi$ is a topo-isomorphy. 
	
	Now consider the induced factor map $\psi\colon X\big/BP(\pi)\to Y$. Since the corresponding actions are assumed to be minimal we know from \cite[Corollary 7.9]{auslander1988minimal} that $\psi$ is equicontinuous, if and only if $RP(\pi)\subseteq R(\phi)$.
	From Theorem \ref{the:MEimpliesRPisRBPisPisBPisICER} we know that the mean equicontinuity of $\pi$ implies $RP(\pi)=BP(\pi)=R(\phi)$. This shows $\psi$ to be equicontinuous. 
\end{proof}

\subsection{Decomposition for weakly mean equicontinuous actions}

	Note that in the proof of Theorem \ref{the:decompositionMinimal} the minimality was only used to show that $\psi$ is equicontinuous. 
	Thus any mean equicontinuous factor map $\pi\colon X\to Y$ decomposes as $\pi=\psi\circ \phi$ with $\phi\colon X\to X\big/BP(\pi)$ topo-isomorphic and $\psi\colon X\big/BP(\pi)\to Y$. In this general context we have the following. 

\begin{proposition}\label{pro:MandBanachDistalForPSI}
	Consider a mean equicontinuous factor map $\pi\colon X\to Y$. The induced factor map $\psi\colon X\big/BP(\pi)\to Y$ satisfies (M) and is Banach distal. 
\end{proposition}
\begin{proof}	
	Denote $\phi\colon X\to X\big/BP(\pi)$. 
	To show that $\psi$ is Banach distal consider $(y,y')\in BP(\psi)$. Choose $x,x'\in X$ such that $(\phi(x), \phi(x'))=(y,y')$. From $(y,y')\in R(\psi)$, we observe $(x,x')\in R(\pi)$.
Since $\phi$ is Banach proximal we observe that $(x,x')$ is Banach proximal from Theorem \ref{the:topo_iso_and_Banach_prox}. Thus we have that $(x,x')\in BP(\pi)=R(\pi)$ and hence $y=\phi(x)=\phi(x')=y'$. This shows $\psi$ to be Banach distal. 

	We use Proposition \ref{pro:M1characterization} to show (M) and consider a convergent sequence $(y_n,y_n')$ in $R(\psi)$ with a limit $(y,y)$ with $y\in X\big/BP(\pi)$. 
	We need to show that $(y_n,y_n')$ is asymptotically Banach proximal. 
	Let $(x_n,x_n')$ be a sequence in $X^2$ with  $(\phi(x_n),\phi(x_n'))=(y_n,y_n')$. Since $(y_n,y_n')\in R(\psi)$, we observe that $(x_n,x_n')\in R(\pi)$.	
	For any cluster point $(x,x')$ of $(x_n,x_n')$ we observe that $(\phi(x),\phi(x'))=(y,y)$ and hence $(x,x')\in R(\phi)=BP(\pi)$. Since $\pi$ is mean equicontinuous, we observe from Proposition \ref{pro:M1characterization} that 
	$(x_n,x_n')$ is asymptotically Banach proximal. 
	In particular, by Proposition \ref{pro:pushingDownBPandaBP} $(y_n,y_n')$ is asymptotically Banach proximal as the image of $(x_n,x_n')$. 
\end{proof}

\begin{theorem}\label{the:decompositionWME}
	Any mean equicontinuous factor map $\pi\colon X\to Y$ between weakly mean equicontinuous actions decomposes as $\pi=\psi\circ \phi$ into a topo-isomorphic factor map $\phi\colon X\to X\big/BP(\pi)$ and an equicontinuous factor map $\psi\colon X\big/BP(\pi)\to Y$.
\end{theorem}
\begin{proof}
	It remains to show that $\psi$ is equicontinuous. Note that we know from Proposition \ref{pro:MandBanachDistalForPSI} that $\psi$ is Banach distal. Thus by Theorem \ref{the:char_equicontinuity} it suffices to show that $\psi$ is mean equicontinuous.
	Let $(y_n,y_n')$ be a convergent sequence in $R(\psi)$ and denote $(y,y')$ for its limit. 
	Since Proposition \ref{pro:MandBanachDistalForPSI} also yields that $\psi$ satisfies (M), we observe from Proposition \ref{pro:M1characterization} that $(y_n,y_n')$ is asymptotically Banach proximal, whenever $(y,y')$ is Banach proximal. Now assume that $(y_n,y_n')$ is asymptotically Banach proximal. 
	Consider a sequence $(x_n,x_n')$ in $X^2$ with $(\phi(x_n),\phi(x_n'))=(y_n,y_n')$. 
	Since $(y_n,y_n')\in R(\psi)$, we observe that $(x_n,x_n')\in R(\pi)$.   
	Let $(x,x')$ be a cluster point of $(x_n,x_n')$ and note that by the continuity of $\phi$ we have $(\phi(x),\phi(x'))=(y,y')$. 
	Since $\phi$ is topo-isomorphic and $X$ is weakly mean equicontinuous, Theorem \ref{the:charTOPOISOMWeakMeanEqui} yields that $(x_n,x_n')$ is asymptotically Banach proximal. 
	Thus $(x,x')$ is the limit of an asymptotically Banach proximal sequence in $R(\pi)$ and the mean equicontinuity of $\pi$ yields that $(x,x')\in BP(\pi)$. 
	We thus observe the Banach proximality of $(y,y')$ from Proposition \ref{pro:pushingDownBPandaBP}. 
	This shows $\psi$ to be mean equicontinuous. 
\end{proof}

\subsection{Uniqueness of the Decomposition}

\begin{lemma}\label{lem:pushingandPullingproxanddist}
	Consider factor maps $\phi_i\colon X\to Y_i$ and $\psi_i\colon Y_i\to Z$ with $\psi_1\circ \phi_1=\psi_2\circ \phi_2$ and  $R(\phi_1)\cap R(\phi_2)=\Delta_X$. 
	If $\psi_1$ is distal, then $\phi_2$ is distal. 
\end{lemma}
\begin{proof}
	Consider $(x,x')\in R(\phi_2)$. Since $\psi_1\circ \phi_1=\psi_2\circ\phi_2$, we observe $(x,x')\in R(\phi_2)\subseteq R(\psi_2\circ \phi_2)=R(\psi_1\circ\phi_1)$ and hence $(\phi_1(x),\phi_1(x'))\in R(\psi_1)$. 
	Since $\psi_1$ is assumed to be distal, $(\phi_1(x),\phi_1(x'))$ must be distal or trivial.
	If $\phi_1(x)=\phi_1(x')$, then $(x,x')\in R(\phi_1)\cap R(\phi_2)=\Delta_X$ and hence $x=x'$. 
	If $(\phi_1(x),\phi_1(x'))$ is distal, then also $(x,x')$ is distal, since images of proximal pairs are proximal.
	This shows that $\phi_2$ is distal. 
\end{proof}

\begin{theorem}\label{the:uniqueness_factorization}
	Consider factor maps $\phi_i\colon X\to Y_i$ 
	and $\psi_i\colon Y_i\to Z$ that satisfy $\psi_1\circ \phi_1=\psi_2\circ \phi_2$. If the $\phi_i$ are proximal and the $\psi_i$ are distal, then there exists a conjugacy $\iota\colon Y_1\to Y_2$ such that $\phi_2=\iota \circ \phi_1$ and $\psi_1=\psi_2\circ \iota$.   
\end{theorem}
\begin{proof}
Consider the invariant closed equivalence relation $R(\phi):=R(\phi_1)\cap R(\phi_2)$. There are canonical factor maps $\phi\colon X\to Y$, $\psi\colon Y\to Z$ and $\pi_i\colon Y\to Y_i$ with 
	$\phi_i=\pi_i\circ \phi$ and 
	$\psi=\psi_i\circ \pi_i$. 
Since the $\phi_i$ are proximal and $\phi_i=\pi_i\circ \phi$ it is straight forward to observe that the $\pi_i$ are proximal. 
Note that $\psi_i\circ \pi_i=\psi$ and that the $\psi_i$ are distal. To apply Lemma \ref{lem:pushingandPullingproxanddist} to these maps we need to show that $R(\pi_1)\cap R(\pi_2)=\Delta_Y$. For this let $(y,y')\in R(\pi_1)\cap R(\pi_2)$ and consider $(x,x')\in X^2$ with $(\phi(x),\phi(x'))=(y,y')$. Clearly $\phi_i(x)=\pi_i(y)=\pi_i(y')=\phi_i(x')$ and hence $(x,x')\in R(\phi_1)\cap R(\phi_2)=R(\phi)$.
	Thus $y=\phi(x)=\phi(x')=y'$. 
	This shows the $\pi_i$ to be distal. 
	
	We have shown that the maps $\pi_i$ are distal and proximal and hence conjugacies. This allows to consider the conjugacy $\iota:=\pi_2\circ \pi_1^{-1}$, which clearly satisfies $\phi_2=\iota \circ \phi_1$ and $\psi_1=\psi_2\circ \iota$. 
\end{proof}

	Recall that any equicontinuous factor map is distal. Furthermore, recall from Theorem \ref{the:characterization_Banach_proximal_factor_maps} that any topo-isomorphic factor map is proximal. 
	
\begin{corollary}
	If a factor map $\pi\colon X\to Z$ decomposes as $\pi=\psi_i\circ \phi_i$ into topo-isomorphic factor maps $\phi_i\colon X\to Y_i$ and equicontinuous factor maps $\psi_i\colon Y_i\to Z$, then there exists a conjugacy $\iota\colon Y_1\to Y_2$ such that $\phi_2=\iota \circ \phi_1$ and $\psi_1=\psi_2\circ \iota$. 
\end{corollary}

\section{Combining factor maps}\label{sec:combiningFactorMaps}

It is well known that the composition of two distal factor maps is a distal factor map and easy to see that the composition of topo-isomorphic factor maps is topo-isomorphic and that the composition of Banach distal factor maps is Banach distal. Furthermore, it is well known that the composition of equicontinuous factor maps is always distal, but does not need to be equicontinuous \cite{auslander1988minimal}.  
	Since any mean equicontinuous and distal factor map is equicontinuous by Theorem \ref{the:char_equicontinuity} we observe that the composition of mean equicontinuous factor maps does not need to be mean equicontinuous.

	In the light of the decomposition theorems above it is natural to ask what happens, if one composes topo-isomorphic factor maps with equicontinuous ones. We next show that in the context of weakly mean equicontinuous actions an equicontinuous factor map after a topo-isomorphic one always yields a mean equicontinuous factor map. Unfortunately, it remains open, whether a similar statement holds also for minimal actions. 
	
\begin{theorem}\label{the:WMEcomposingtiandEq}
	Consider factor maps $\phi\colon X\to Y$ and $\psi\colon Y\to Z$ between weakly mean equicontinuous actions. 
	If $\phi$ is topo-isomorphic and $\psi$ is mean equicontinuous, then $\pi:=\psi\circ \phi\colon X\to Z$ is mean equicontinuous. 
\end{theorem}
\begin{proof}
	Let $(x_n,x_n')$ be a convergent sequence in $R(\pi)$ and denote $(x,x')$ for its limit. 
	Denote $(y_n,y_n'):=(\phi(x_n),\phi(x_n'))$ and $(y,y'):=(\phi(x),\phi(x'))$ and note that $(y_n,y_n')$ is a sequence in $R(\psi)$ with limit $(y,y')\in R(\psi)$. 
	
	If $(x_n,x_n')$ is asymptotically Banach proximal, then by Proposition \ref{pro:pushingDownBPandaBP} we observe $(y_n,y_n')$ to be asymptotically Banach proximal and the mean equicontinuity of $\psi$ implies that $(y,y')\in BP(\psi)$. 
	Recall that $\phi$ is topo-isomorphic. 
	From Theorem \ref{the:topo_iso_and_Banach_prox} we observe that $(x,x')$ is Banach proximal. 
	
	If $(x,x')$ is Banach proximal, then Proposition \ref{pro:pushingDownBPandaBP} yields that $(y,y')\in BP(\psi)$. Since $\psi$ is assumed to be mean equicontinuous, we have that $(y_n,y_n')$ is asymptotically Banach proximal. Recall that $\phi$ is a topo-isomorphy and that $X$ is weakly mean equicontinuous. From Theorem \ref{the:charTOPOISOMWeakMeanEqui} we observe that 
	$(x_n,x_n')$ is asymptotically Banach proximal. 	
\end{proof}

From Theorem \ref{the:decompositionWME} we observe the following. 

\begin{corollary}
	A factor map $\pi$ between weakly mean equicontinuous actions is mean equicontinuous, if and only if it is the composition $\pi=\psi\circ \phi$ of a topo-isomorphic factor map $\phi$ and an equicontinuous factor map $\psi$. 	
\end{corollary}
	
The following example shows that the composition of a topo-isomorphic factor map after an equicontinuous factor map does not need to be mean-equicontinuous, even if the actions are assumed to be minimal and uniquely ergodic and we have $G=\mathbb{Z}$. It furthermore shows that there exist factor maps that are Banach distal, but not distal. 

\begin{example}\label{exa:topo-isometricAfterEquicont}
	Details on the notions used in the following example can be found in \cite{downarowicz2005survey}. 
	Let $Z:=\mathbb{Z}_2$ be the dyadic odometer. 
	We define $f\colon Z\to \{0,1\}$ to be the characteristic function of $\bigcup_{n\in 2 \mathbb{N}_0} 2^{n+1}\mathbb{Z}_2+2^{n}$. 
	Note that $f$ has precisely one discontinuity at $0$. 
	We denote $\gamma$ for the two sided sequence given by the restriction of $f$ to $\mathbb{Z}$. 
	Since $f$ is not continuous on $\mathbb{Z}$ we have that $\gamma$ is not a Toeplitz sequence. However, since $f_{z'}:=z\mapsto f(z+z')$ is continuous for some $z'\in \mathbb{Z}_2\setminus \mathbb{Z}$ we observe that $Y:=\overline{\mathbb{Z}.\gamma}$ is a regular Toeplitz subshift. In particular, it is minimal and uniquely ergodic. Since there is only one discontinuity of $f$ we have that 
	$Z=\mathbb{Z}_2$ is the maximal equicontinuous factor of $Y$ under some factor map $\psi\colon Y\to Z$. Assuming that $\psi(\gamma)=0$, the property that $f$ has a single discontinuity implies that $\psi$ it is two-to-one on $\mathbb{Z}$, but one-to-one elsewhere. 
	Denote $\gamma'$ for the sequence given by flipping the $0^{\text{th}}$ coordinate in $\gamma$ to $1$. 
	Then $\gamma$ and $\gamma'$ are in a fibre (mapped to $0$) and all other two-to-one fibres are given by shifting $\gamma$ and $\gamma'$ accordingly.
	Note that $\psi$ is clearly a topo-isomorphy, as it is one-to-one onto the full subset $\mathbb{Z}_2\setminus \mathbb{Z}$.
	
	It is well known that the Toeplitz subshift
	 constructed above is also given by the constant length substitution $0\mapsto 01; 1\mapsto 00$. 
	Furthermore, it is folklore 
	that the Thue-Morse substitution  $0\mapsto 01; 1\mapsto 10$ yields a minimal uniquely ergodic subshift $X$ that is a two-to-one extension of $Y$ via the factor map 
	$\phi\colon X\to Y$ mapping 
	$x\in X$ to the sequence $y$ given by $y_n:=1$, if $x_{n}=x_{n+1}$ and $y_n:=0$, else. 
	Let us denote $\overline{x}$ for the element wise negation, i.e.\ $\overline{x}_n:=1-x_n$ for all $n$. Clearly for any $x\in X$ we have that $x$ and $\overline{x}$ share one fibre. Since $d(x,\overline{x})=1$, we observe $\phi$ to be an equicontinuous extension. 
	
	Now consider $\pi:=\psi\circ \phi\colon X\to Z$. 
	We observe that for $z\in \mathbb{Z}_2\setminus \mathbb{Z}$ that $\psi^{-1}(z)$ is a singleton, whose preimage $\pi^{-1}(z)$ contains a pair of the form $(x,\overline{x})$. Since $D(x,\overline{x})=1$, we have that $D$ is constantly one on all two-to-one fibres. 
	Now consider $z\in \mathbb{Z}$. Since $\mathbb{Z}$ is an orbit in $\mathbb{Z}_2$ we assume w.l.o.g\ that $z=0$. 
	We know that $\gamma$ and $\gamma'$ are contained in $\psi^{-1}(z)$. Now recall that $\gamma$ and $\gamma'$ differ at precisely in there $0^{\text{th}}$ coordinate. Thus the elements of $\pi^{-1}(z)$ must be of the form $\alpha.\beta$, $\overline{\alpha}.\beta$, $\alpha.\overline{\beta}$, and $\overline{\alpha}\overline{\beta}$ with one sided infinite sequences $\alpha$ and $\beta$ to the left and right, respectively.
	Note that $\alpha.\beta$ and $\alpha.\overline{\beta}$ are proximal and hence $\pi$ is not distal and in particular not equicontinuous. 
	Furthermore, from this description of the four-to-one fibres it is straight forward to observe that $D$ is constantly $1$ on $R(\pi)$ and in particular Banach distal. In particular, we observe from Theorem \ref{the:char_equicontinuity} that $\pi$ cannot be mean equicontinuous. 	
\end{example}

\begin{acknowledgement}
This article was funded by the Deutsche Forschungsgemeinschaft (DFG, German Research Foundation) – 530703788.
	The author would like to express sincere thanks to Henrik Kreidler for raising the question regarding the possibility of a relativized notion of mean equicontinuity. The author is grateful to Henrik Kreidler, Julian Hölz, Jochen Glück, Gabriel Fuhrmann, and Maik Gröger for their insightful discussions on the topic of mean equicontinuity, which significantly contributed to the development of this work. The author wishes to thank Henrik Kreidler for introducing him to Example \ref{exa:Dnotcontinuous}. Finally, the author wishes to thank Tobias Jäger for introducing him to the Thue-Morse example, which proved to be an invaluable resource in the course of the research. 
\end{acknowledgement}

\bibliography{Ref}
\bibliographystyle{ijmart}

\end{document}